\documentclass[9pt]{amsproc}
\usepackage[utf8]{inputenc}
\usepackage{amsmath, amsfonts, amsthm, amssymb, amstext, amscd, multicol, multirow, makecell, supertabular, lscape, graphicx, tikz, bm, algpseudocode, algorithm, hyperref}

\usepackage[all,2cell,ps]{xy}

\newcommand{\Belyi}{Bely\u{\i}~}
\newcommand{\tors}{\text{tors}}
\newcommand{\ord}{\text{ord}}

\newtheorem*{quest}{Research Question}
\newtheorem{theorem}{Theorem}
\newtheorem{prop}[theorem]{Proposition}
\newtheorem{coro}[theorem]{Corollary}
\newtheorem{lem}[theorem]{Lemma}

\title{Critical Points of Toroidal Bely\u{\i} Maps}

\author[T.~Asmara]{Tesfa Asmara}
\address{Pomona College, 610 North College Avenue, Claremont CA 91711}
\email{tgac2020@mymail.pomona.edu}

\author[E.~Goins]{Edray Herber Goins}
\address{Pomona College, 610 North College Avenue, Claremont CA 91711}
\email{edray.goins@pomona.edu}

\author[E.~Imathiu-Jones]{Erik Imathiu-Jones}
\address{California Institute of Technology}
\curraddr{506 Bienville Boulevard, Dauphin Island, AL 36528}
\email{eimathiu@caltech.edu}

\author[M.~Maalouf]{Maria Maalouf}
\address{California State University at Long Beach} 
\email{maria.maalouf@student.csulb.edu}

\author[I.~Robinson]{Isaac Robinson}
\address{Harvard University}
\curraddr{524 Eliot Street, Milton, MA 02186}
\email{isaac\_robinson@college.harvard.edu}

\author[S.~Spaulding]{Sharon Sneha Spaulding}
\address{University of Connecticut}
\curraddr{46 Orchard Road, Windsor CT 06095}
\email{sharon.spaulding@uconn.edu}

\date{\today}

\begin{document}

\begin{abstract}
A \Belyi map $\beta: \mathbb{P}^1(\mathbb{C}) \to \mathbb{P}^1(\mathbb{C})$ is a rational function with at most three critical values; we may assume these values are $\{ 0, \, 1, \, \infty \}$.  Replacing $\mathbb{P}^1$ with an elliptic curve $E: \ y^2 = x^3 + A \, x + B$, there is a similar definition of a \Belyi map $\beta: E(\mathbb{C}) \to \mathbb{P}^1(\mathbb{C})$.  Since $E(\mathbb{C}) \simeq \mathbb T^2(\mathbb {R})$ is a torus, we call $(E, \beta)$ a Toroidal \Belyi pair. There are many examples of \Belyi maps $\beta: E(\mathbb{C}) \to \mathbb P^1(\mathbb{C})$ associated to elliptic curves; several can be found online at LMFDB. Given such a Toroidal \Belyi map of degree $N$, the inverse image $G = \beta^{-1} \bigl( \{ 0, \, 1, \, \infty \} \bigr)$ is a set of $N$ elements which contains the critical points of the \Belyi map. In this project, we investigate when $G$ is contained in $E(\mathbb{C})_{\text{tors}}$. This is work done as part of the Pomona Research in Mathematics Experience (NSA H98230-21-1-0015).
\end{abstract}


\maketitle

%===================================================================================================

\section{Introduction}

Let $S$ be a compact, connected Riemann Surface.  For example, $S = \mathbb P^1(\mathbb C) \simeq S^2(\mathbb R)$ may be the Riemann Sphere, or $S = E(\mathbb C) \simeq (\mathbb R/\mathbb Z) \times (\mathbb R/\mathbb Z)$ may be a torus associated to an elliptic curve $E$.  It is well-known that $S$ is a curve, that is, can be defined by a single equation $f(x,y) = 0$ for some polynomial $f(x,y) = \sum_{ij} a_{ij} \, x^i \, y^j$ having complex coefficients $a_{ij}$.  Andr{\'e} Weil proved in 1956 that one can choose these coefficients to lie in a number field if there exists a meromorphic function $\beta: S \to \mathbb P^1(\mathbb C)$ with at most three critical values; Gennadi\u{\i} Vladimirovich Bely\u{\i} proved the converse to this in 1979.  For this reason, we call $\beta$ a \Belyi map.  One can always choose the critical values of a \Belyi map $\beta: S \to \mathbb P^1(\mathbb C)$ to lie in $\{ 0, \, 1, \, \infty \} \subseteq \mathbb P^1(\mathbb C)$.  We denote $\Gamma = \beta^{-1} \bigl( \{ 0, \, 1, \, \infty \} \bigr) \subseteq S$ as the quasi-critical points of $\beta$.  

There are many examples of \Belyi maps $\beta: S \to \mathbb P^1 (\mathbb C)$ associated to a Riemann surface $S$.  Several can be found online at the $L$-Series and Modular Forms Database (LMFDB).  In this work, we are primarily interested in \Belyi maps associated to elliptic curves: we call $(E, \beta)$ a Toroidal \Belyi pair.   We consider how the quasi-critical points $\Gamma \subseteq E(\mathbb C)$ of a Toroidal \Belyi pair $(E, \beta)$ interact with the Group Law $\oplus$ on an elliptic curve.  As a motivating example, consider the elliptic curve $E: y^2 = x^3 + 1$ and the \Belyi map $\beta(x,y) = (1-y)/2$. The set of quasi-critical points is $\Gamma = \{(0, 1), \, (0, -1), \, O_E\} \simeq Z_3$, a subgroup of $S = E(\mathbb C)$.

We are motivated by two primary research questions. Given a Toroidal \Belyi pair $(E, \beta)$, when does its set of quasi-critical points $ \Gamma$ form a subgroup of $(E (\mathbb C), \oplus)$?  If $\Gamma$ is a group, then its elements must have finite order. When are the quasi-critical points torsion elements in $E (\mathbb C)$ -- regardless of $\Gamma$ being a group?  Our main results are as follows. 

{\vskip 0.1in \noindent \textbf{Theorem.} \textit{Say $(X,\phi)$ is a Toroidal \Belyi pair, and denote $G = \phi^{-1} \bigl( \{ 0, \, 1, \, \infty \} \bigr)$ as the set of quasi-critical points. Take $\beta = \phi \circ \psi$, where $\psi: E \to X$ is any non-constant isogeny, and denote $\Gamma = \beta^{-1} \bigl( \{ 0, \, 1, \, \infty \} \bigr)$.
\begin{enumerate}
\item $(E,\beta)$ is a Toroidal \Belyi pair.
\item $\Gamma$ is contained in the torsion in $E(\mathbb{C})$ whenever $G$ is contained in the torsion in $X(\mathbb{C})$. 
\item $\Gamma$ is a group whenever $G$ is group.
\end{enumerate}} \vskip 0.1in}

\noindent As a direct consequence, we find the following.

{\vskip 0.1in \noindent \textbf{Corollary.} \textit{There are infinitely many \Belyi pairs where the set of quasi-critical points forms a group.} \vskip 0.1in}

We would like to thank our research advisor Edray Goins for leading and guiding us throughout the preparation and completion of this project. We would also like to acknowledge Alex Barrios, Rachel Davis, and the other PRiME students for a productive and inclusive environment, John Voight for the Toroidal \Belyi pair data in LMFDB, and all the mathematicians who have contributed to LMFDB over the years. We would also like to thank the NSA for funding our project (H98230-21-1-0015).

%===================================================================================================

\section{Background and Notation}

We begin by introducing definitions and known results relevant for our main research questions. 

\subsection{Groups}

A group is a pair $(G,\oplus)$ which consists of a non-empty set $G$ and a binary operation $\oplus: G \times G \to G$ such that $G$ contains an identity element $O$, every element $P \in G$ has an inverse element $[-1]P \in G$, and $\oplus$ is associative. A group  is said to be abelian if $\oplus$ is also commutative.  As an example of an abelian group, consider the pair $(Z_n, +)$ where $Z_n=\{0, 1, \dots , n-1\}$ and $+$ denotes addition modulo $n$.

The order of a group is the number of elements in $G$.  A group is said to be finite if the set $G$ is finite. The order of $P \in G$ is the smallest positive integer $n$ such that $[n]P=O$, where $[n]P$ denotes $P \oplus P \oplus \dots \oplus P$ for exactly $n$ summands $P$.  If no such $n$ exists, then an element is said to have infinite order. Otherwise, an element has finite order and is called a torsion element. 

A subset $H \subseteq G$ is said to be a subgroup of $G$ if $H$ forms a group under $\oplus$.  More generally, we may consider the subgroup \emph{generated} by the elements of $H$: this is the smallest subgroup of $G$ containing $H$. 

We may also define maps between groups. Given two groups $(G, \oplus)$ and $(\Gamma, \star)$, a group homomorphism is a map $\psi: G \to \Gamma$ such that $\psi(P \oplus Q)= \psi(P) \star \psi(Q)$ for $P,Q \in G$. The kernel, denoted $\ker(\psi)$, is the set $P \in G$ such that $\psi(P)= O_\Gamma$ where $O_\Gamma$ is the identity element in $\Gamma$; this is a subgroup of $G$.  If $\psi: G \to \Gamma$ is a group homomorphism for which $\ker(\psi) = \{ O_\Gamma \}$ and $\psi$ is surjective, $\psi$ is said to be an isomorphism between $G$ and $\Gamma$.  In this case, we denote $G \simeq \Gamma$. 

The following proposition provides a group theoretic result relevant to this paper.

\begin{prop} \label{thm:lagrange} Let $G$ be finite group and let $P \in G.$ Then the order of $P$ divides the order of $G$. \end{prop}

\subsection{Number Fields}

Let $\nu \in \mathbb C$ be a root of an irreducible polynomial $f(T) = c_n \, T^n + \cdots + c_1 \, T + c_0$ with coefficients $c_k \in \mathbb Q$.  We denote $K = \mathbb Q(\nu)$ as the collection of complex numbers in the form $a_0 + a_1 \, \nu + \cdots + a_{n-1} \, \nu^{n-1}$ where $a_k \in \mathbb Q$. The set $K$ is called a number field.

Say that $s \in \mathbb C$ is the root of a irreducible polynomial $g(T) = d_m \, T^m + \cdots + d_1 \, T + d_0$ with coefficients $d_k \in K$.  We denote $L = K(s)$ as the collection of complex numbers in the form  $b_0 + b_1 \, s + \cdots + b_{m-1} \, s^{m-1}$ where $b_k \in K$.  The set $L$ is called an extension of $K$; note that $L$ is also a number field.

We define an embedding $L$ into $\mathbb C$ fixing $K$ to be that map where we evaluate $s \mapsto s_i$ for some root $s_i \in \mathbb C$ of $g(T)$.  We denote $\text{Emb}(L/K)$ as the collection of embeddings $L \hookrightarrow \mathbb C$ fixing $K$.  (For those who know about Galois groups, we can write $\text{Emb}(L/K) = \text{Gal}\bigl( \overline{\mathbb Q}/L \bigr) / \text{Gal}\bigl( \overline{\mathbb Q}/K \bigr)$ as a collection of cosets.)

\subsection{Riemann Sphere}

We define the extended complex line, denoted by $\mathbb P^1(\mathbb C)$, as the set of complex numbers together with infinity; recall that there is a one-to-one correspondence between the points on the extended complex line and the points on the unit sphere $S^2(\mathbb R)$.  For this reason, we often call $\mathbb P^1(\mathbb C) \simeq S^2(\mathbb R)$ the Riemann Sphere.  It will be useful for us to view $\mathbb P^1(\mathbb C)$ as a non-singular curve of genus 0, that is, the collection of complex points $P = (x,y)$ satisfying $f(x,y) = 0$, where $f(x,y) = y$.

\subsection{Elliptic Curves}

We will now outline some of the general theory of elliptic curves relevant for this paper. An elliptic curve, denoted by $E,$ is a non-singular curve of genus one. In other words, it is a curve generated by an equation $f(x,y) = 0$, where
\begin{equation*} f(x, y) =  y^2 + a_1 \, x \, y + a_3 \, y - (x^3 + a_2 \, x^2 + a_4 \, x + a_6) \end{equation*} 
\noindent and where $a_1$, $a_2$, $a_3$, $a_4$, and $a_6$ are complex numbers.  We denote $E(\mathbb C)$ to be the collection of complex points $P = (x,y)$ on $E$ augmented by a ``point at infinity'' $O_E$. For more details, see \cite[Chapter I.4, page 28]{MR93g:11003}. 

\begin{prop} \label{prop:compositionprop} 
Let $E$ be an elliptic curve over $\mathbb C$.
\begin{enumerate}
\item There exists a binary operation $\oplus$ such that $(E(\mathbb C), \oplus)$ is an abelian group with identity $O_E$.
\item Points $P, Q, R$ on $E(\mathbb C)$ lie on a line if and only if $P \oplus Q \oplus R= O_E$.
\end{enumerate} 
\end{prop}

\noindent For details on the Chord-Tangent method, see \cite[Chapter I.4]{MR93g:11003}.  For details on the collinearity property, see \cite[Chapter III.2, Proposition 2.2, pages 51-52]{MR2514094}.

Similar to homomorphisms between two groups, an isogeny $\psi: E(\mathbb C) \to X(\mathbb C)$ is a group homomorphism between two elliptic curves, that is, $\psi(P\oplus Q) = \psi(P) \oplus  \psi(Q)$ for $P, Q \in E(\mathbb C)$. 

\begin{prop} \label{prop:nonconstantisogeny}
Let $\psi: E(\mathbb C) \to X(\mathbb C)$ be a non-constant isogeny. Then $\psi$ is surjective, and $\ker (\psi)$ is a finite subgroup of $E(\mathbb C)$. 
\end{prop}

\noindent For details, see \cite[Chapter IV, Corollary 4.9]{MR2514094} and \cite[Chapter 2, Theorem 2.3]{MR2514094}. 

Since we consider the points of an elliptic curve as forming a group, we define the order of a point $P \in E(\mathbb C)$ in the same way as we previously defined the order of a group element $P \in G$. In the same way, we define a torsion point as a point of finite order. The set of torsion elements for an elliptic curve $E$ over the complex numbers is denoted $E(\mathbb C)_{\tors}$. 

\begin{prop} \label{prop:torsion}  Let $E$ be an elliptic curve over $\mathbb C$.
\begin{enumerate}
\item $E(\mathbb C) \simeq (\mathbb R / \mathbb Z) \times (\mathbb R / \mathbb Z)$.  In particular, this set of complex points forms a torus. 
\item $E(\mathbb C)_\tors \simeq (\mathbb Q/ \mathbb Z) \times (\mathbb Q/ \mathbb Z).$
\item Assume $G \subseteq E(\mathbb C)_\tors$ is a finite subgroup. Then $G \simeq Z_m \times Z_n$ for some positive integers $m$ and $n$.
\end{enumerate}
\end{prop}

\noindent For details, see \cite[Chapter VI.5, Corollary 5.1.1, page 173]{MR2514094}.

\subsection{\Belyi Maps}

Denote either $S= E(\mathbb C)$ or $S = \mathbb P^1(\mathbb C)$.  Note that in either case, $S$ is a curve generated by an equation $f(x,y) = 0$ for some polynomial $f(x,y)$. (In the projects discussed in this exposition, we will focus on the sphere and the torus, but many of the definitions hold for any compact, connected Riemann surface $S$.  It is well-known that such surfaces may be identified as a curve, that is, generated by an equation $f(x,y) = 0$ for some polynomial $f(x,y)$.)

A meromorphic function is a map $\beta: S \to \mathbb P^1(\mathbb C)$ that is a ratio of two polynomials; denote the function field $\mathcal K(S)$ as the collection of all such functions. For each point $P=(x_0, y_0)$ in $S$, denote $\mathcal O_P \subseteq \mathcal K(S)$ as the collection of meromorphic functions such that $\beta(P) \neq \infty$. For any positive integer $e$, denote 
\[ {M_P}^e = \left \{ \phi \in \mathcal O_P \ \left| \ \begin{aligned} \phi(x,y) & = g(x,y) \cdot f(x,y) \\[0pt] & + \sum_{i+j=e} p_{ij}(x,y) \cdot (x-x_0)^i (y-y_0)^j \end{aligned} \quad \text{for $g, \, p_{ij} \in \mathcal O_P$} \right. \right \}. \]
\noindent  For example, $M_P$ is just the collection of those meromorphic satisfying $\beta(P) = 0$.  Denote the order of $\beta$ at $P$ as the integer
\begin{equation*} \text{ord}_P(\beta) = \begin{cases} e \geq 0 & \text{if $\beta(P) \neq \infty $ and $\beta \in {M_P}^e$ but $\beta \notin {M_P}^{e+1}$, and} \\ e<0 & \text{if $\beta(P) = \infty$ and $1/\beta \in {M_P}^{-e}$ but $1/\beta \notin {M_P}^{1-e}$.} \end{cases} \end{equation*} 
\noindent The ramification index of $\beta$ at $P \in E(\mathbb C)$ is defined as $e_\beta(P)= \ord_P \left[\beta(x,y)- \beta(P)\right]$. (Order and ramification can also be defined via places and valuations; for further details see \cite[Chapter 3.4]{MR2895884}.)

\begin{prop} \label{crit_pt_defn}
Let $S$ be a compact, connected Riemann surface defined by a polynomial $f(x,y)$.  Given meromorphic function $\beta: S \to \mathbb P^1(\mathbb C)$, the ramification index $e_\beta(P) \geq 2$ at a point $P \in S$ if and only if 
\[ \dfrac {\partial f}{\partial x}(P) \ \dfrac {\partial \beta}{\partial y}(P) - \dfrac {\partial f}{\partial y}(P) \ \dfrac {\partial \beta}{\partial x}(P) = 0. \]
\end{prop}

\noindent To see why, note that, for any function $g \in \mathcal O_P$, we have a series expansion around $P = (x_0, y_0)$ in the form
\begin{equation*} \begin{aligned}
\bigl[ \beta(x,y) - \beta(P) \bigr] & + \left[ g(P) \ \dfrac {\partial f}{\partial x}(P) - \dfrac {\partial \beta}{\partial x}(P) \right] (x-x_0) + \left[ g(P) \ \dfrac {\partial f}{\partial y}(P) - \dfrac {\partial \beta}{\partial y}(P) \right] (y-y_0) \\[5pt] & = g(x,y) \cdot f(x,y) + \sum_{i+j=2} p_{ij}(x,y) \cdot (x-x_0)^i (y-y_0)^j \in {M_P}^2 \end{aligned} \end{equation*}
\noindent for some $p_{ij} \in \mathcal O_P$.  This means $\beta(x,y) - \beta(P) \in {M_P}^2$ if and only if we can find $q = g(P) \in \mathbb C$ such that
\begin{equation*} \begin{aligned} \dfrac {\partial \beta}{\partial x}(P) & = q \cdot \dfrac {\partial f}{\partial x}(P)\\[5pt] \dfrac {\partial \beta}{\partial y}(P) & = q \cdot \dfrac {\partial f}{\partial y}(P) \end{aligned} \qquad \iff \qquad \dfrac {\partial \beta}{\partial x}(P) \ \dfrac {\partial f}{\partial y}(P) - \dfrac {\partial \beta}{\partial y}(P) \ \dfrac {\partial f}{\partial x}(P) = 0. \end{equation*} 

A point $P \in S$ for which the conditions in Proposition \ref{crit_pt_defn} hold is called a critical point. A critical value $q \in \mathbb P^1(\mathbb C)$ is a number $q = \beta(P)$ for some critical point $P$. A point $Q \in S$ is a quasi-critical point if $\beta(Q)= \beta(P)$ for some critical point $P$.  The degree of a meromorphic function $\beta: S \to \mathbb P^1(\mathbb C)$ is the size of the inverse image $\beta^{-1}(\{q\})$ for any $q \in \mathbb P^1(\mathbb C)$ that is not a critical value. 

A \Belyi pair $(S, \beta)$ is a Riemann surface $S$ along with a meromorphic function $\beta: S \to \mathbb P^1(\mathbb C)$ with at most three critical values.  We can -- and do -- choose these values to be contained in $\{0, 1, \infty\} \subseteq \mathbb P^1(\mathbb C)$.  There are two types of \Belyi pairs which we are specifically interested in for this exposition. A \Belyi map $\gamma: \mathbb P^1(\mathbb C) \to \mathbb P^1(\mathbb C)$ is  dynamical if $\gamma \bigl( \{ 0, 1, \infty \} \bigr) \subseteq \{ 0, 1, \infty \}$.  A Toroidal \Belyi pair $(E, \beta)$ consists of an elliptic curve $E$ and a \Belyi map $\beta: E(\mathbb C) \to \mathbb P^1(\mathbb C)$.  A Toroidal \Belyi pair is defined to be imprimitive if it can be written as a non-trivial composition $\beta = \gamma \circ \phi \circ \psi$ for some isogeny $\psi: E(\mathbb C) \to X(\mathbb C)$, meromorphic function $\phi \in \mathcal K \bigl( X(\mathbb C) \bigr)$, and dynamical \Belyi map $\gamma \in \mathcal K \bigl( \mathbb P^1(\mathbb C) \bigr)$. 

\begin{prop} \label{prop:sumramind}
Let $S$ be a compact, connected Riemann surface of genus $g(S)$.  Let $(S, \beta)$ be a \Belyi pair with critical values contained in $\{ 0, \, 1, \, \infty \} \subseteq \mathbb P^1(\mathbb C)$ with ramification indices $e_P = e_\beta(P)$ as well as preimages $B= \beta^{-1}(\{0\})$, $W= \beta^{-1}(\{1\})$, and $F= \beta^{-1}(\{\infty\})$.  Then the quasi-critical points are contained in the disjoint union $B \cup W \cup F$, and we have the identity
\begin{equation*} \deg(\beta)= \sum_{P \in B} e_P = \sum_{P \in W} e_P = \sum_{P \in F} e_P = |B| + |W| + |F| + \bigl( 2 \, g(S) - 2 \bigr). \end{equation*}
\end{prop}

\noindent For details see \cite[Proposition 2.6, Chapter II.2, page 24]{MR2514094}.

\subsection{Divisors}

Continue to denote either $S = \mathbb P^1(\mathbb C)$ or $S = E(\mathbb C)$, although many of the definitions which follow hold for any compact, connected Riemann surface $S$.  

A divisor is a formal sum $D = \sum_{P \in S} n_P (P)$, with $n_P \in \mathbb{Z}$ and all but finitely many being zero. The degree of a divisor is the integer $\deg D = \sum_{P \in S} n_P$. Denote the collection of degree 0 divisors as $ \text{Div}^0(S)$.  Observe that $(\text{Div}^0(S), +)$ is an abelian group under addition.  Indeed, given two divisors $D_1 = \sum_{P \in S} c_P \, (P)$ and $D_2 = \sum_{P \in S} d_P \, (P)$ as well as integers $a$ and $b$, define $a \, D_1 + b \, D_2 =  \sum_{P \in S} n_P \, (P)$ in terms of the integers $n_P = a \, c_P + b \, d_P$.  For details, see \cite[Chapter II.3, page 27]{MR2514094}.

Let $\beta: S \to \mathbb P^1(\mathbb C)$ be a meromorphic function which is not identically zero.  We can associate to $\beta$ a divisor of the form $\text{div}(\beta) = \sum_{P \in S} n_P \, (P)$ where $n_P = \text{ord}_P(\beta)$.   A divisor $D$ is principal if $D = \text{div}(\beta)$ for some meromorphic function $\beta$. The degree of a principal divisor is zero, and the collection of principal divisors forms a subgroup of $\text{Div}^0(S)$. 

Divisors on elliptic curves are intimately related to the group law.
\begin{prop} \label{prop:principaldivisorprop}
Let $E$ be an elliptic curve over $\mathbb C$. A divisor $D = \sum_{P \in S} n_P \, (P)$ on $S = E(\mathbb{C})$ is principal if and only if $\sum_{P \in S} n_P = 0$ in $\mathbb Z$ and $\bigoplus_{P \in S} [n_P]P = O_E$ in $S$.
\end{prop} 
\noindent For details, see \cite[Chapter III.3, Corollary 3.5, page 63]{MR2514094}

Say $\phi: S \to \mathbb P^1(\mathbb C)$ is a meromorphic function which is not identically zero.  There is a group homomorphism $\phi^\ast: \text{Div}^0 \bigl( \mathbb P^1(\mathbb C) \bigr) \to \text{Div}^0(S)$, called the pullback of $\phi$, which is defined as follows: If $D = \sum_{q \in \mathbb P^1(\mathbb C)} n_q \, (q)$ is a divisor of degree 0 on $\mathbb P^1(\mathbb C)$, then $\phi^\ast D = \sum_{P \in S} m_P \, (P)$ is a divisor of degree 0 on $S$, where $m_P  = e_\phi(P) \cdot n_{\phi(P)}$.

\begin{prop} \label{prop:tesfa} Denote either $S = \mathbb P^1(\mathbb C)$ or $S = E(\mathbb C)$. 
\hfill
\begin{enumerate}
\item Assume that $D = \text{div}(\gamma)$ is a principal divisor on $\mathbb P^1(\mathbb C)$.  Then $\phi^\ast D = \text{div}(\beta)$ is a principal divisor on $S$ where $\beta = \gamma \circ \phi$.
\item For any meromorphic function $\phi: S \to \mathbb P^1(\mathbb C)$ which is not identically zero, we have the pullback
\begin{equation*} \phi^\ast \bigl( (0) - (\infty) \bigr) = \sum_{P \in \phi^{-1}(\{0\})} n_P \, (P) \quad - \sum_{P \in \phi^{-1}(\{\infty\})} n_P \, (P) \qquad \text{in terms of $n_P =\ord_P(\phi)$.} \end{equation*}
\end{enumerate} \end{prop}
\noindent For more details, see
\cite[Chapter II.3, Proposition 3.6, page 29]{MR2514094} and \cite[Chapter II.3, Example 3.5, page 29]{MR2514094}.  The following proposition shows that divisors behave similarly to logarithms.

\begin{prop} \label{prop:tesfa2} 
Let $\beta_1, \, \beta_2: S \to \mathbb P^1(\mathbb C)$ be meromorphic functions which are not identically zero.
\begin{enumerate}
\item $\text{div}(\beta_1^a \cdot \beta_2^b) = a \cdot \text{div}(\beta_1) + b \cdot \text{div}(\beta_2)$ for any integers $a$ and $b$.
\item $\text{div}(\beta_1) = \text{div}(\beta_2)$ if and only if $\beta_1 = k \cdot \beta_2$ for some nonzero $k \in \mathbb C$.
\end{enumerate} \end{prop}
\noindent For more details, see \cite[Chapter II.3, Proposition 3.1, page 28]{MR2514094}.

%===================================================================================================

\section{Initial Investigations}

\subsection{Motivating Examples and Questions}

Given a \Belyi map on an elliptic curve, we can look at its quasi-critical points. We can look at how the quasi-critical points interact with the elliptic curve group law. To get a better idea of this, let us look at a couple of examples.

Consider the Toroidal \Belyi pair $(E, \beta)$ with $E$ the curve defined by $f(x,y) = y^2 - (x^3 + 1)$ and $\beta(x,y) = (1-y)/2$. We can compute the critical points $P = (x,y)$ of $\beta$ by finding when the following function vanishes: $(\partial f/\partial x) \, (\partial \beta/\partial y) - (\partial f/\partial y) \, (\partial \beta/\partial x) = (3/2) \, x^2$. We find that the critical points are $\{(0, 1), \, (0, -1), \, O_E\}$, which is isomorphic to $Z_3$. 

As a second example, consider the Toroidal \Belyi pair $(E, \beta)$ with $E$ the curve defined by $f(x,y) = y^2 - (x^3 - x)$ and $\beta(x,y) = x^2$. Again, we can compute the critical points $P = (x,y)$ of $\beta$ by finding when the following function vanishes: $(\partial f/\partial x) \, (\partial \beta/\partial y) - (\partial f/\partial y) \, (\partial \beta/\partial x) = -4 \, x \, y$. We find that the collection of critical points is the set $\{(-1, 0), \, (0,0), \, (+1, 0), \, O_E\}$, which is isomorphic to $Z_2 \times Z_2$.

Following our observations from these examples, there are two main questions that arise.
\begin{quest} \label{question1} Say $(E, \beta)$ is a Toroidal \Belyi pair, and denote the collection $\Gamma = \beta^{-1} \bigl( \{ 0, \, 1, \, \infty \} \bigr)$ as the collection of quasi-critical points. When does $\Gamma$ form a subgroup of $\bigl( E(\mathbb C), \oplus \bigr)$? By Proposition \ref{thm:lagrange}, the elements in $\Gamma$ must be points with finite order whenever $\Gamma$ is a group.  When are the points in $\Gamma$ torsion elements in $E(\mathbb C)$, regardless of $\Gamma$ being a group?
\end{quest}

\subsection{Searching for Examples} \label{tesfasubsection}

We began our exploratory analysis by calculating a number of examples. We started by pulling examples of Toroidal \Belyi pairs $(X, \phi)$ from the $L$-Series and Modular Forms Database (LMFDB) \cite{lmfdb}, computing quasi-critical points $P \in \phi^{-1} \bigl( \{ 0, \, 1, \, \infty \} \bigr)$ and their ramification indices $e_P = e_\phi(P)$; some data can be found in Table \ref{table:initial_examples}.  We soon realized that we would need to write a computer program to systematically compute the quasi-critical points and their orders.  Simply put, even though the Toroidal \Belyi pair $(X, \phi)$ was defined over a number field $K = \mathbb Q(\nu)$, the quasi-critical points in general would be defined over an extension $L = K(s)$.  Figure \ref{fig:goins} contains a diagram of the various function fields where we need to carry out our computations.

\begin{figure}[htb]
\begin{equation*} \xymatrix  @!=0.2in {
{} & {} & {} & {} & {} & {} & {\mathbb C(X) = \mathcal K \bigl( X(\mathbb C) \bigr)}  \ar@{-}[d] \ar@{-}[dll] \\
{} & {} & {} & {} & {L(X) = \mathcal K \bigl( X(L) \bigr)} \ar@{-}[d] \ar@{-}[dll] & {} & {\mathbb C(x) = \mathcal K \bigl( \mathbb P^1(\mathbb C) \bigr)} \ar@{-}[dll] \ar@{-}[d] \\
{} & {} & {K(X) = \mathcal K \bigl( X(K) \bigr)} \ar@{-}[d] & {} & {L(x) = \mathcal K \bigl( \mathbb P^1(L) \bigr)} \ar@{-}[d] \ar@{-}[dll] & {} & {\mathbb C} \ar@{-}[dll] \\
{} & {} & {K(x) = \mathcal K \bigl( \mathbb P^1(K) \bigr)} \ar@{-}[d] & {} & {L = K(s)} \ar@{-}[dll] & {} & {} & {} \\ 
{} & {} & {K = \mathbb Q(\nu)} \ar@{-}[dll] & {} & {} & {} & {} \\
{\mathbb Q} & {} & {} & {} & {} & {} & {} } \end{equation*}
\caption{Function Fields related to a Toroidal \Belyi Pair $(X, \phi)$ defined over $K$}
\label{fig:goins}
\end{figure}
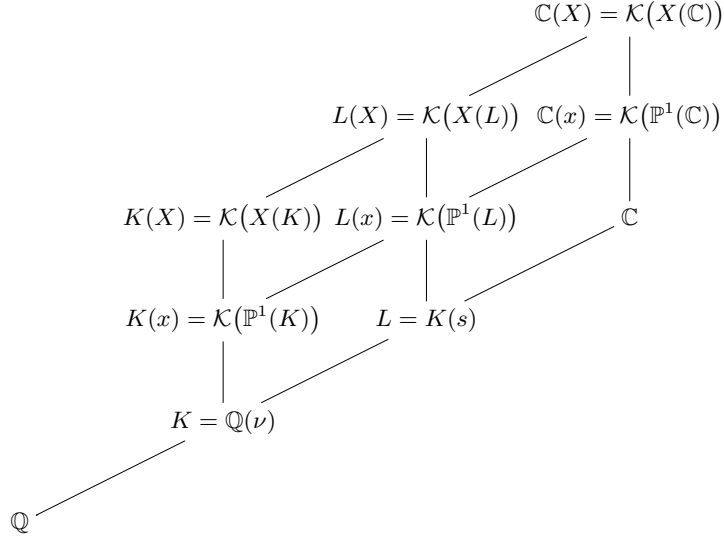

\begin{landscape} \begin{table}[!htbp]
 \centering
\begin{tabular}{llll}

{Elliptic Curve $X$} & {\Belyi Map $\phi(x,y)$} & {Quasi-Critical $P \in \phi^{-1} \bigl( \{ 0, \, 1, \, \infty \} \bigr)$} & {Ramification $e_\phi(P)$} \\[5pt] \hline \hline & & & \\[3pt]

\href{https://www.lmfdb.org/EllipticCurve/Q/36/a/4}{$y^2 = x^3 + 1$} & \href{https://beta.lmfdb.org/Belyi/3T1/3/3/3/a}{$\dfrac {y+1}{2}$} &  {$(0, \pm 1)$, $O_E$} & {3, 3, 3} \\[11pt] \hline & & & \\[3pt]

\href{https://www.lmfdb.org/EllipticCurve/Q/32/a/3}{$y^2 = x^3 - x$} & \href{https://beta.lmfdb.org/Belyi/4T1/4/4/2.2/a}{$x^2$} & {$(\pm 1,0)$, $(0,0)$, $O_E$} & {2, 2, 4, 4} \\[8pt]

\href{https://www.lmfdb.org/EllipticCurve/Q/48/a/6}{$y^2 = x^3 + x^2 + 16 \, x + 180$} & \href{https://beta.lmfdb.org/Belyi/4T5/4/4/3.1/a}{$\dfrac {4 \, y + x^2 + 56}{108}$} & {$(22, -108)$, $(-2,12)$, $(4,-18)$, $O_E$} & {1, 3, 4, 4} \\[11pt] \hline & & & \\[3pt]

% \href{https://www.lmfdb.org/EllipticCurve/Q/50/a/3}{$y^2 + x \, y + y = x^3 -  x - 2$} & \href{https://beta.lmfdb.org/Belyi/5T3/5/4.1/4.1/a}{$\dfrac {\begin{aligned} (x & + 3) \, y - 2 \, x^2 \\ & + 2 \, x + (2 + 4 \, \sqrt{-1}) \end{aligned}}{8 \, \sqrt{-1}}$} & {--} & {--} \\[8pt]

\href{https://www.lmfdb.org/EllipticCurve/Q/50/b/4}{$\begin{aligned} y^2 & + x \, y + y \\ & = x^3 + x^2 + 22 \, x - 9 \end{aligned}$} & \href{https://beta.lmfdb.org/Belyi/5T4/5/5/2.2.1/a}{$\dfrac {x \, y + 3 \, x^2 + 3 \, x + 63}{64}$} & {$\begin{aligned} & (9, -37), (-1\pm2\sqrt{-5}, 3), \\ &(1,-5), O_E \end{aligned}$} & {1, 2, 2, 5, 5} \\[16pt]

\href{https://www.lmfdb.org/EllipticCurve/Q/150/c/3}{$y^2 + x \, y = x^3 - 28 \, x + 272$} & \href{https://beta.lmfdb.org/Belyi/5T4/5/5/3.1.1/a}{$\dfrac {(x+13) \, y + 3 \, x^2 + 4 \, x + 220}{432}$} & {$\begin{aligned} & \bigl( -13 \pm 3 \sqrt{-15}, \, 2 \, (37 \pm 3\sqrt{-15}) \bigr), \\ & (2, -16), (-4,20), O_E \end{aligned}$} & {1, 1, 3, 5, 5} \\[16pt]

\href{https://www.lmfdb.org/EllipticCurve/Q/75/a/2}{$y^2 + y = x^3 + x^2 + 2 \, x + 4$} & \href{https://beta.lmfdb.org/Belyi/5T4/5/5/3.1.1/b}{$\dfrac {(x+7) \, y - 5 \, x^2 - 2 \, x + 15}{27}$} & {$\begin{aligned} & \bigl( (13\pm 3\sqrt{-15})/2, \ 13\pm 6\sqrt{-15} \bigr), \\ & (-1,-2), (2,4), O_E \end{aligned}$} & {1, 1, 3, 5, 5} \\[16pt]

\href{https://www.lmfdb.org/EllipticCurve/Q/900/e/2}{$y^2 = x^3 - 120 \, x + 740$} & \href{https://beta.lmfdb.org/Belyi/5T5/5/3.2/3.2/a}{$\dfrac {(x+5)\, y + 162}{324}$} & {$(-11, \pm 27)$, $(4, \pm18)$, $O_E$} & {2, 2, 3, 3, 5} \\[16pt]

% \href{}{$\begin{aligned} y^2 & = x^3 \\ & + 60 \, \bigl( 1699 - 1256 \sqrt{6} \bigr) \, x \\ & + 80 \, \bigl( 2237141 - 924804 \sqrt{6} \bigr) \end{aligned}$} & \href{https://beta.lmfdb.org/Belyi/5T5/5/3.2/4.1/a}{$\dfrac {\begin{aligned} -&( 12 + 5 \sqrt{6} ) \, x \, y - 30 \, ( 11 + 4 \sqrt{6} ) \, x^2 \\ & + 10 \, ( 36 - 35 \sqrt{6} ) \, y - 120 \, ( 19 + 16 \sqrt{6} ) \, x \\ & \quad + 1464 \, ( 23897 + 3636 \sqrt{6} ) \end{aligned}}{47775744}$} & {--} & {--} \\[16pt]

\href{https://www.lmfdb.org/EllipticCurve/Q/400/a/1}{$y^2 = x^3 + 5 \, x + 10$} & \href{https://beta.lmfdb.org/Belyi/5T5/5/4.1/4.1/a}{$\dfrac {(x-5) \, y + 16}{32}$} & {$(6, \pm 16)$, $(1, \pm 4)$, $O_E$} & {1, 1, 4, 4, 5} \\

\end{tabular} 
\vskip 0.5in
\caption{Examples of Toroidal \Belyi Pairs $(X, \phi)$ with Quasi-Critical Points and Ramification Indices}
\label{table:initial_examples}
\end{table} \end{landscape}

The outline of our methodology is as follows:

\begin{algorithm}
\renewcommand{\thealgorithm}{} % This removes algorithm number
	\caption{Compute Examples} 
	\begin{algorithmic}[1]
	\State Choose Toroidal \Belyi pair $(X,\phi)$ defined over a number field $K = \mathbb Q(\nu)$ from LMFDB.
	\State Viewing $\phi \in \mathcal K \bigl( X(\mathbb C) \bigr)$, compute $\text{div}(\phi)$, $\text{div}(\phi - 1)$, and $\text{div}(1/\phi)$ to find the quasi-critical points over $\mathbb C$.
	\State Compute the smallest number field $L = K(s)$ containing all the quasi-critical points.
	\State Extend the function field $K(X) = \mathcal K \bigl( X(K) \bigr)$ of the elliptic curve to $L$, that is, $L(X) = \mathcal K \bigl( X(L) \bigr)$.
	\State Working over $L$, compute the divisors, quasi-critical points, and their orders.
	\State Identify the smallest subgroup that contains all of the quasi-critical points.
	\State Write divisors, quasi-critical points, and the group generated by them to an external file.
	\end{algorithmic} 
\end{algorithm}

\noindent The above was all implemented using a combination of \texttt{python} and \texttt{sage} on the cloud computing provider \texttt{CoCalc}. Computational power limited our ability to calculating the number field $L$ in many cases. As a result, we implemented a ``time-out,'' a specified time interval for computing the number field after which the example would be skipped. Still, we were able to compute 13 examples of Toroidal \Belyi pairs $(X, \phi)$ for which the quasi-critical points are all torsion, that is, $\phi^{-1} \bigl( \{ 0, \, 1, \, \infty \} \bigr) \subseteq X(\mathbb C)_\tors$.  A summary of the results can be found in Table \ref{table:erik}, and the complete list of examples can be found in Table \ref{table:examples}.  The full code can be found in our \texttt{GitHub} repository \cite{PRiME2021}.

\begin{table}[!htbp]
\centering
\begin{tabular}{|c||c|c|c|}
\hline {$\deg(\phi)$} & \makecell[c]{Total from \\ LMDFB} & \makecell[c]{Total Number of \\ Successfully Processed} & \makecell[c]{Number with Quasi-Critical \\ Points All Torsion} \\ \hline \hline
3 & 1 & {1 (100\%)} & {1 (100\%)} \\
4 & 2 & {2 (100\%)} & {2 (100\%)} \\
5 & 7 & {7 (100\%)} & {1 (14\%)} \\
6 & 35 & {29 (83\%)} & {7 (24\%)} \\
7 & 73 & {15 (21\%)} & {0 (0\%)} \\
8 & 94 & {30 (32\%)} & {2 (7\%)} \\
9 & 39 & {23 (59\%)} & {0 (0\%)} \\
\hline
Totals & 251 & {107 (43\%)} & {13 (12\%)} \\
\hline
\end{tabular}
\vskip 0.2in
\caption{Processing of Toroidal \Belyi pairs ($X, \phi)$ from LMFDB.  We wish to find a number field $L$ such that $\phi^{-1} \bigl( \{ 0, \, 1, \, \infty \} \bigr) \subseteq X(L)$ -- which we may not be able to do in \texttt{sage}.}
\label{table:erik}
\end{table}

\begin{landscape} \begin{table}[!htbp]
\centering
\begin{tabular}{c|ccc} 
{\texttt{LMFDB} Label} & {Elliptic Curve $X$} & {\Belyi Map $\phi(x,y)$} & {Group Generated by $G$} \\[3pt] \hline \hline & & & \\[-3pt]
{\href{https://beta.lmfdb.org/Belyi/3T1/3/3/3/a}{3T1-3\_3\_3-a}} & {\href{https://beta.lmfdb.org/EllipticCurve/Q/36/a/4}{$y^2 = x^3 + 1$}} & {$\dfrac {1-y}{2}$} & {$Z_3$} \\[7pt] \hline & & & \\[-2pt]
{\href{https://beta.lmfdb.org/Belyi/4T1/4/4/2.2/a}{4T1-4\_4\_2.2-a}} & {\href{https://beta.lmfdb.org/EllipticCurve/Q/32/a/3}{$y^2 = x^3 - x$}} & {$1-x^2$} & {$Z_2 \times Z_2$} \\[2pt]
{\href{https://beta.lmfdb.org/Belyi/4T5/4/4/3.1/a}{4T5-4\_4\_3.1-a}} & {\href{https://beta.lmfdb.org/EllipticCurve/Q/48/a/6}{$y^2 = x^3 + x^2 + 16 \, x + 180$}} & {$\dfrac {4 \, y + x^2 + 56}{108}$} & {$Z_8$} \\[6pt] \hline & & & \\[-3pt]
{\href{https://beta.lmfdb.org/Belyi/5T4/5/5/3.1.1/a}{5T4-5\_5\_3.1.1-a}} & {\href{https://beta.lmfdb.org/EllipticCurve/Q/150/c/3}{$y^2 + x \, y = x^3 - 28 \, x + 272$}} & {$\dfrac {(x+13) \, y + 3 \, x^2 + 4 \, x + 220}{432}$} & {$Z_2 \times Z_{10}$} \\[7pt] \hline & & & \\[-2pt]
{\href{https://beta.lmfdb.org/Belyi/6T1/6/2.2.2/3.3/a}{6T1-6\_2.2.2\_3.3-a}} & {\href{https://beta.lmfdb.org/EllipticCurve/Q/36/a/4}{$y^2 = x^3 + 1$}} & {$-x^3$} & {$Z_2 \times Z_6$} \\[3pt]
{\href{https://beta.lmfdb.org/Belyi/6T4/3.3/3.3/3.3/a}{6T4-3.3\_3.3\_3.3-a}} & {\href{https://beta.lmfdb.org/EllipticCurve/Q/36/a/2}{$y^2 = x^3 - 15 \, x + 22$}} & {$\dfrac {8 \, (x-2)^2 - (x^2-4 \, x+ 7) \, y}{16 \, (x-2)^2}$} & {$Z_6$} \\[7pt]
{\href{https://beta.lmfdb.org/Belyi/6T5/6/6/3.1.1.1/a}{6T5-6\_6\_3.1.1.1-a}} & {\href{https://beta.lmfdb.org/EllipticCurve/Q/36/a/4}{$y^2 = x^3 + 1$}} & {$\dfrac {(1-y) \, (3+y)}{4}$} & {$Z_2 \times Z_6$} \\[7pt]
{\href{https://beta.lmfdb.org/Belyi/6T6/6/6/2.2.1.1/a}{6T6-6\_6\_2.2.1.1-a}} & {\href{https://beta.lmfdb.org/EllipticCurve/Q/72/a/5}{$y^2 = x^3 + 6 \, x - 7$}} & {$\dfrac {(x-1)^3}{27}$} & {$Z_2 \times Z_4$} \\[7pt]
{\href{https://beta.lmfdb.org/Belyi/6T7/4.2/4.2/3.3/a}{6T7-4.2\_4.2\_3.3-a}} & {\href{https://beta.lmfdb.org/EllipticCurve/Q/7056/q/3}{$y^2 = x^3 - 10731 \, x + 408170$}} & {$\dfrac {11907 \, (x-49)}{(x-7)^3}$} & {$Z_2 \times Z_4$} \\[7pt]
{\href{https://beta.lmfdb.org/Belyi/6T12/5.1/5.1/3.3/b}{6T12-5.1\_5.1\_3.3-b}} & {\href{https://beta.lmfdb.org/EllipticCurve/Q/15/a/5}{$y^2 + x \, y + y = x^3 + x^2 - 10 \, x - 10$}} & {$27 \, \dfrac {(x + 4) \, (2 \, x^2 - 2 \, x - 13) - (x+1)^2 \, y}{(x^2 - x - 11)^3} $} & {$Z_2 \times Z_8$} \\[7pt]
{\href{https://beta.lmfdb.org/Belyi/6T12/5.1/5.1/5.1/a}{6T12-5.1\_5.1\_5.1-a}} & {\href{https://beta.lmfdb.org/EllipticCurve/Q/20/a/4}{$y^2 = x^3 + x^2 + 4 \, x + 4$}} & {$-16 \, \dfrac{(x^2-2 \, x - 4) \, y + 8 \, (x+1)}{(x-4) \, x^5}$} & {$Z_6$} \\[7pt] \hline & & & \\[-2pt]
{\href{https://beta.lmfdb.org/Belyi/8T2/4.4/4.4/2.2.2.2/a}{8T2-4.4\_4.4\_2.2.2.2-a}} & {\href{https://beta.lmfdb.org/EllipticCurve/Q/64/a/4}{$y^2 = x^3 + x$}} & {$\dfrac{(x+1)^4}{8 \, x \, (x^2 + 1)}$} & {$Z_2 \times Z_4$} \\[7pt]
{\href{https://beta.lmfdb.org/Belyi/8T7/8/8/2.2.1.1.1.1/a}{8T7-8\_8\_2.2.1.1.1.1-a}} & {\href{https://beta.lmfdb.org/EllipticCurve/Q/32/a/3}{$y^2 = x^3 - x$}} & {$x^4$} & {$Z_2 \times Z_4$} \\
\end{tabular}

\vskip 0.5in
\caption{Examples of Toroidal \Belyi Pairs $(X, \phi)$ with Quasi-Critical Points $G = \phi^{-1} \bigl( \{ 0, \, 1, \, \infty \} \bigr) \subseteq X(\mathbb C)_\tors$}
\label{table:examples}
\end{table} \end{landscape}

\nocite{sagemath}
\nocite{lmfdb}

%===================================================================================================

\section{Main Results}

\noindent In this section, we list the main results from our research.

We observed that, given a Toroidal \Belyi pair $(E, \beta)$, we could construct another Toroidal \Belyi pair $(X, \phi)$ where $X = E$ is the same elliptic curve but $\phi(x,y) = \beta \bigl( (x,y) \oplus P_0 \bigr)$ could be the translate by point $P_0 \in E(\mathbb C)$.  If there is any hope of the collection of quasi-critical points being torsion, then we would need to limit the possibilities of quasi-critical points by choosing certain translates of \Belyi maps.  The following proposition explains one way we can do this.  

\begin{theorem} \label{thm:tesfa}
Say $(E, \beta)$ is a Toroidal \Belyi pair, with $N = \deg(\beta)$, and denote
\[ Q_0 = \bigoplus_{P \in \beta^{-1}(\{0\})} [e_P] P = \bigoplus_{P \in \beta^{-1}(\{1\})} [e_P] P = \bigoplus_{P \in \beta^{-1}(\{\infty\})} [e_P] P. \] 
\noindent Then $\beta$ can be normalized, that is, there exists $P_0 \in E(\mathbb C)$ satisfying $[N] P_0 = Q_0$ such that $\beta \bigl( (x,y) \oplus P_0 \bigr) = f(x,y) / g(x,y)$ for two polynomials $f, \, g \in \mathcal K \bigl( E(\mathbb C) \bigr)$ with divisors
\[ \begin{aligned}
\text{div}(f) & = \sum_{P \in B} e_P \, (P) - N \, (O_E) \\
\text{div}(f - g) & = \sum_{P \in W} e_P \, (P) - N \, (O_E) \\
\text{div}(g) & = \sum_{P \in F} e_P \, (P) - N \, (O_E)
\end{aligned} \qquad \text{where} \qquad \begin{aligned}
B & = \beta^{-1}(\{0\}) \ominus P_0, \\[5pt] 
W & = \beta^{-1}(\{1\}) \ominus P_0, \\[5pt] 
F & = \beta^{-1}(\{\infty\}) \ominus P_0.
\end{aligned} \]
\end{theorem}

\noindent Observe that if $\beta: E(\mathbb C) \to \mathbb P^1(\mathbb C)$ is a normalized Toroidal \Belyi map, then we may choose $P_0 = Q_0 = O_E$.  Hence the quasi-critical points, i.e. the points in $B$, $W$, and $F$, have relations involving their ramification indices.  For instance, if we have a normalized \Belyi map with $\beta^{-1}( \{ q \} ) = \{ P \}$ for some $q \in \mathbb P^1(\mathbb C)$, then $q \in \{ 0, \, 1, \, \infty \}$ must be a critical value, and so $[N] P = O_E$ since $e_\beta(P) = N$ is the degree of the \Belyi map.

\begin{proof}
Denote $\phi(x,y) = \beta \bigl( (x,y) \oplus P_0 \bigr)$. Then, observe that $\phi^{-1}(\{q\}) = \beta^{-1}(\{q\}) \ominus P_0$, for any $q \in \mathbb P^1(\mathbb C)$. Recall that $e_P = e_\beta(P) = e_\phi(P \ominus P_0)$. Then, by Proposition \ref{prop:tesfa}, we have the principal divisors
\begin{equation*} \begin{aligned} \text{div}(\phi) & = \sum_{P \in B} e_P (P) - \sum_{P \in F} e_P(P) \\[5pt] \text{div}(\phi - 1) & = \sum_{P \in W} e_P (P) - \sum_{P \in F} e_P(P) \end{aligned} \qquad \text{where} \qquad \begin{aligned} B & = \beta^{-1}(\{0\}) \ominus P_0 =  \phi^{-1}(\{0\}), \\[5pt] W & = \beta^{-1}(\{1\}) \ominus P_0 =  \phi^{-1}(\{1\}), \\[5pt] F & = \beta^{-1}(\{\infty\}) \ominus P_0 =  \phi^{-1}(\{\infty\}). \end{aligned} \end{equation*}

\noindent Then, it follows from Proposition \ref{prop:principaldivisorprop} that

\begin{equation*}
\begin{aligned}
\left( \bigoplus_{P \in B} [e_P] P \right) \ominus \left( \bigoplus_{P \in F} [e_P] P \right) \,= \left( \bigoplus_{P \in W} [e_P] P \right) \ominus \left( \bigoplus_{P \in F} [e_P] P \right) = O_E.
\end{aligned}
\end{equation*}

\noindent The statement for $Q_0$ follows. To show that $P_0$ exists as claimed, consider the map $\psi: E(\mathbb C) \to E(\mathbb C)$ defined by $\psi(P)= [N] P$. Proposition \ref{prop:nonconstantisogeny} asserts that $\psi$ is surjective, hence the statement for $P_0$ follows. We will show that $f, g$ exist as claimed by showing that $D_1 = \sum_{P \in B} e_P \, (P) - N \, (O_E)$ and $D_2 = \sum_{P \in F} e_P \, (P) - N \, (O_E)$ are principal divisors. First, consider $D_1$. Then, $\text{deg}(D_1) = \sum_{P \in B} e_P \, - N \,= N - N \,= 0$ by Proposition \ref{prop:sumramind}; and, by the definition of $Q_0 = [N] P_0$,

\[\begin{aligned} \left(\bigoplus\limits_{P \in B} [e_P] P\right) \oplus [-N] O_E \,= \bigoplus\limits_{P \in \beta^{-1}(\{0\})} [e_P] (P \ominus P_0) \,= [N] P_0 \ominus [N] P_0 \, = O_E. \end{aligned}\]

\noindent It follows from Proposition \ref{prop:principaldivisorprop} that there exists $f \in \mathcal K \bigl( E(\mathbb C) \bigr)$ such that $\text{div}(f) = D_1$. By a similar argument, there exists $g \in \mathcal K \bigl( E(\mathbb C) \bigr)$ such that $\text{div}(g) = D_2$. Now observe that

\[ \begin{aligned} \text{div} (f/g) & = \text{div}(f) - \text{div}(g) = \left(\sum_{P \in B} e_P \, (P) - N \, (O_E)\right) - \left(\sum_{P \in F} e_P \, (P) - N \, (O_E)\right)\\ & = \text{div}(\phi). \end{aligned} \]

\noindent Therefore, Proposition \ref{prop:tesfa2} asserts that $\phi = k \cdot f/g$, for some constant $k$. Substituting $k \cdot f$ as $f$, if necessary, we see that $\phi = f/g$. Consider $\text{div}(f-g)$. Using that $\phi = f/g$, substitute in $f= \phi \cdot g$ to see that

\[ \begin{aligned} \text{div}(f-g) & = \text{div}\bigl( g \cdot (\phi - 1) \bigr) = \text{div}(g) + \text{div}(\phi - 1) \\[3pt] & = \left(\sum_{P \in F} e_P \, (P) - N \, (O_E)\right) + \left(\sum_{P \in W} e_P (P) - \sum_{P \in F} e_P(P)\right) \\[3pt] & = \sum_{P \in W} e_P \, (P) - N \, (O_E). \end{aligned}\]

\end{proof}

\noindent Recall from Table \ref{table:examples} that we have 13 examples of Toroidal \Belyi pairs $(E, \beta)$ where the collection of quasi-critical points are torsion.  The phenomenon that we have seen so far where the quasi-critical points forms a group occurs infinitely often. 

\begin{theorem} \label{thm:maria}
Say $(X,\phi)$ a Toroidal \Belyi pair, and denote $G = \phi^{-1} \bigl( \{ 0, \, 1, \, \infty \} \bigr)$ as the set of quasi-critical points. Take $\beta = \phi \circ \psi$, where $\psi: E \to X$ is any non-constant isogeny, and denote $\Gamma = \beta^{-1} \bigl( \{ 0, \, 1, \, \infty \} \bigr)$.
\begin{enumerate}
\item $(E,\beta)$ is a Toroidal \Belyi pair.
\item $\Gamma$ is contained in the torsion in $E(\mathbb{C})$ whenever $G$ is contained in the torsion in $X(\mathbb{C})$. 
\item $\Gamma$ is a group whenever $G$ is group.
\end{enumerate}
\end{theorem}

\noindent In order to prove this theorem, we will proceed with the proofs of numerous lemmata. Observe that $\Gamma = \{ P \in E(\mathbb{C}) ~|~ \psi(P) \in G\} = \psi^{-1}(G)$; this will be useful in the proofs.

%Lemma 1
\begin{lem}
 $(E, \beta)$ is a Toroidal \Belyi pair.
\end{lem}

\begin{proof}
Assume by way of contradiction that $\beta= \phi \circ \psi$ is not a \Belyi map. By assumption, there exists a point $P \in E(\mathbb C)$ such that $\beta(P)=q\not\in \{0, 1, \infty\}$ is a critical value. Since $q$ is a critical value, $e_\beta(P) \geq 2$. 
However %proposition that will be added by Sharon: 
$e_\beta(P)= e_\phi(\psi(P))$, it follows that $e_\phi(Q) \geq 2$ for some $Q= \psi(P) \in X(\mathbb{C})$. Then $Q$ is a critical point for $\phi$ with value $q = \beta(P) = \phi(Q)$. %$\phi(Q)= \phi(\psi(P))= \beta(P)= q$.
Then, $\phi$ has a critical value $q \not\in \{0, 1, \infty\}$, which is a contradiction. Therefore, $\beta$ is a \Belyi map. 
\end{proof}

\begin{lem}
If $G\subseteq X(\mathbb C)_\tors$ then $\Gamma \subseteq E(\mathbb{C})_{\tors}$. 
\end{lem}

\begin{proof}
Take $Q=\psi(P) \in G$ with $P\in \Gamma$. Since $G\subseteq X(\mathbb C)_\tors$, then there exists a positive integer $n$ such that $[n]Q=O_X$. Since $\psi$ is a group homomorphism, then $[n]Q= [n]\psi(P)= \psi([n]P)$. It follows that $\psi([n]P)=O_X$.
Thus, $[n]P \in \ker(\psi)$, 
which is shown to be finite in Proposition \ref{prop:nonconstantisogeny}.
By Proposition \ref{thm:lagrange}, 
there exists a positive integer $m$ such that $[m]R=O_E$ for any $R\in \ker(\psi)$. Denoting $N=m \, n$, we have $[N] P = [m] \bigl( [n] P \bigr) = O_E$, showing $P \in E(\mathbb C)_\tors$. Thus, $\Gamma \subseteq E(\mathbb{C})_{\tors}$.
\end{proof}

\begin{lem}
Suppose $(G, \oplus)$ is a group. Then $\Gamma$ is a subgroup of $(E(\mathbb C), \oplus)$.
\end{lem}

\begin{proof}
To show that $\Gamma$ is a subgroup of $(E(\mathbb C), \oplus)$, we show that (i) $\Gamma$ is a non-empty set and (ii) that $\Gamma$ is closed under differences. For (i), $\psi(O_E) = O_X \in G$ because $(G, \oplus)$ is a group and $\psi$ is a group homomorphism, so $O_E \in \psi^{-1}(G) = \Gamma$. For (ii), consider $\psi(P), \psi(Q) \in G$ where $P,Q \in \Gamma$. Since $(G, \oplus)$ is a group, we have $\psi(P \ominus Q) = \psi(P) \ominus \psi(Q) \in G$, which means that $P \ominus Q \in \Gamma$. Thus, $\Gamma$ is a subgroup of $(E(\mathbb C), \oplus)$.
\end{proof}

\begin{coro} \label{coro:maria}
There are infinitely many imprimitive Toroidal \Belyi pairs where the set of quasi-critical points forms a group.
\end{coro} 
    
\begin{proof} Consider $X: y^2 = x^3 + 1$ and the \Belyi map $\phi(x,y) = (1-y)/2$.  We have seen that the quasi-critical points, namely $G = \phi^{-1} \bigl( \{ 0, \, 1, \, \infty \} \bigr) = \bigl \{ (0,-1), \, (0,1), \, O_E \bigr \} \simeq Z_3$, forms a group.  Theorem \ref{thm:maria} asserts that $(E, \beta)$ forms a Toroidal \Belyi pair for any non-constant isogeny $\psi: E(\mathbb C) \to X(\mathbb C)$ where $\Gamma = \beta^{-1} \bigl( \{ 0, \, 1, \, \infty \} \bigr)$ forms a group.  Since there are infinitely many such isogenies, the result follows. \end{proof}

\noindent It appears that Corollary \ref{coro:maria} can generate infinitely many examples of Toroidal \Belyi pairs $(E, \beta)$ where the quasi-critical points forms a group, but we can only show this for imprimitive such pairs.  Table \ref{table:imprimitive} shows how many of the examples from Table \ref{table:examples} are actually associated to imprimitive Toroidal \Belyi pairs.

%===================================================================================================

\section{Future Work}

Upon completion of this project, we have three main goals for future related work. Firstly, we would like to modify the Sage code so that we can process more examples. Next, we would like to know if there are more examples with quasi-critical points and which cannot be explained by our main theorem. We have 13 examples where the quasi-critical points are torsion, and we have one example that can be explained by our main theorem.  Finally, we would like to create a web page where we can host the data found over the summer so that others may easily access and view our results.

%===================================================================================================

\bibliographystyle{amsplain}

%===================================================================================================

\begin{landscape} \begin{table}[!htbp]
\centering
\begin{tabular}{c|cccc} 
{\texttt{LMFDB} Label} & {Elliptic Curve $E$} & {Toroidal \Belyi $\beta(x,y)$} & {\Belyi $\gamma(z)$} & {Meromorphic $\phi(x,y)$} \\[8pt] \hline & & & & \\[4pt]
{\href{https://beta.lmfdb.org/Belyi/4T1/4/4/2.2/a}{4T1-4\_4\_2.2-a}} & {\href{https://beta.lmfdb.org/EllipticCurve/Q/32/a/3}{$y^2 = x^3 - x$}} & {$1-x^2$} & \href{https://beta.lmfdb.org/Belyi/2T1/2/2/1.1/a}{$4 \, z \, (1-z)$} & {$\dfrac {x+1}{2}$} \\[10pt]
{\href{https://beta.lmfdb.org/Belyi/6T1/6/2.2.2/3.3/a}{6T1-6\_2.2.2\_3.3-a}} & {\href{https://beta.lmfdb.org/EllipticCurve/Q/36/a/4}{$y^2 = x^3 + 1$}} & {$-x^3$} & \href{https://beta.lmfdb.org/Belyi/2T1/2/2/1.1/a}{$4 \, z \, (1-z)$} & {$\dfrac {1-y}{2}$} \\[10pt]
{\href{https://beta.lmfdb.org/Belyi/6T5/6/6/3.1.1.1/a}{6T5-6\_6\_3.1.1.1-a}} & {\href{https://beta.lmfdb.org/EllipticCurve/Q/36/a/4}{$y^2 = x^3 + 1$}} & {$\dfrac {(1-y) \, (3+y)}{4}$} & \href{https://beta.lmfdb.org/Belyi/2T1/2/2/1.1/a}{$4 \, z \, (1-z)$} & {$\dfrac {3+y}{4}$} \\[10pt]
{\href{https://beta.lmfdb.org/Belyi/6T6/6/6/2.2.1.1/a}{6T6-6\_6\_2.2.1.1-a}} & {\href{https://beta.lmfdb.org/EllipticCurve/Q/72/a/5}{$y^2 = x^3 + 6 \, x - 7$}} & {$\dfrac {(x-1)^3}{27}$} & \href{https://beta.lmfdb.org/Belyi/3T1/3/3/1.1.1/a}{$z^3$} & {$\dfrac {x-1}{3}$} \\[10pt]
{\href{https://beta.lmfdb.org/Belyi/6T7/4.2/4.2/3.3/a}{6T7-4.2\_4.2\_3.3-a}} & {\href{https://beta.lmfdb.org/EllipticCurve/Q/7056/q/3}{$y^2 = x^3 - 10731 \, x + 408170$}} & {$\dfrac {11907 \, (x-49)}{(x-7)^3}$} & \href{https://beta.lmfdb.org/Belyi/3T2/3/2.1/2.1/a}{$z^2 \, (3 - 2 \, z)$} & {$\dfrac {63}{x-7}$} \\[10pt]
{\href{https://beta.lmfdb.org/Belyi/8T2/4.4/4.4/2.2.2.2/a}{8T2-4.4\_4.4\_2.2.2.2-a}} & {\href{https://beta.lmfdb.org/EllipticCurve/Q/64/a/4}{$y^2 = x^3 + x$}} & {$\dfrac{(x+1)^4}{8 \, x \, (x^2 + 1)}$} & \href{https://beta.lmfdb.org/Belyi/4T2/2.2/2.2/2.2/a}{$\dfrac {(z^2 + 1)^2}{4 \, z^2}$} & {$\dfrac {\sqrt{2} \, x}{y}$} \\[10pt]
{\href{https://beta.lmfdb.org/Belyi/8T7/8/8/2.2.1.1.1.1/a}{8T7-8\_8\_2.2.1.1.1.1-a}} & {\href{https://beta.lmfdb.org/EllipticCurve/Q/32/a/3}{$y^2 = x^3 - x$}} & {$x^4$} & \href{https://beta.lmfdb.org/Belyi/4T1/4/4/1.1.1.1/a}{$z^4$} & {$x$} \\[10pt]
\hline \hline \\[5pt]
{\texttt{LMFDB} Label} & {Elliptic Curve $E$} & {Toroidal \Belyi $\beta(x,y)$} & {\Belyi $\phi(x,y)$} & {Isogeny $\psi(x,y)$} \\[8pt] \hline & & & & \\[0pt]
{\href{https://beta.lmfdb.org/Belyi/6T4/3.3/3.3/3.3/a}{6T4-3.3\_3.3\_3.3-a}} & {\href{https://beta.lmfdb.org/EllipticCurve/Q/36/a/2}{$y^2 = x^3 - 15 \, x + 22$}} & {$\dfrac {\begin{aligned} 8 & \, (x-2)^2 \\ & - (x^2-4 \, x+ 7) \, y \end{aligned}}{16 \, (x-2)^2}$} & {$\dfrac {1-y}{2}$} & {$\left( \dfrac {x^2 - 2 \, x - 3}{4 \, (x-2)}, \ \dfrac {(x^2 - 4 \, x + 7) \, y}{8 \, (x-2)^2} \right)$} \\[10pt]
\end{tabular}

\vskip 0.5in
\caption{Toroidal \Belyi Pairs $(E, \beta)$ from Table \ref{table:examples} which are Imprimitive: $\beta = \gamma \circ \phi \circ \psi$}
\label{table:imprimitive}
\end{table} \end{landscape}

\end{document}